\documentclass[reqno,10pt,centertags]{amsart} 
\usepackage{amsmath,amsthm,amscd,amssymb,latexsym,esint,upref,stmaryrd,enumerate,color,verbatim,yfonts}
\usepackage{hyperref} 
\newcommand*{\mailto}[1]{\href{mailto:#1}{\nolinkurl{#1}}}

\newcommand{\R}{{\mathbb R}}
\newcommand{\N}{{\mathbb N}}

\newcommand{\C}{{\mathbb C}}

\newcommand{\bbC}{{\mathbb{C}}}

\newcommand{\bbN}{{\mathbb{N}}}

\newcommand{\bbR}{{\mathbb{R}}}
\newcommand{\bbS}{{\mathbb{S}}}

\newcommand{\cB}{{\mathcal B}}
\newcommand{\cC}{{\mathcal C}}
\newcommand{\cD}{{\mathcal D}}

\newcommand{\cF}{{\mathcal F}}

\newcommand{\cH}{{\mathcal H}}

\newcommand{\cL}{{\mathcal L}}
\newcommand{\cM}{{\mathcal M}}

\newcommand{\cR}{{\mathcal R}}
\newcommand{\cS}{{\mathcal S}}

\newcommand{\p}{\prime}

\newcommand{\beq}{\begin{equation}}
\newcommand{\enq}{\end{equation}}

\DeclareMathOperator{\dom}{dom}

\renewcommand{\Re}{\text{\rm Re}}
\renewcommand{\Im}{\text{\rm Im}}
\renewcommand{\ln}{\text{\rm ln}}

\newcommand{\no}{\notag}
\newcommand{\lb}{\label}
\newcommand{\f}{\frac}

\newcommand{\ol}{\overline}

\newcommand{\Oh}{O}

\newcommand{\dott}{\,\cdot\,}

\newcommand{\bi}{\bibitem}

\let\geq\geqslant
\let\leq\leqslant

\DeclareMathOperator*{\LIM}{\text{l.i.m.}}

\makeatletter
\def\theequation{\@arabic\c@equation}

\allowdisplaybreaks 
\numberwithin{equation}{section}

\newtheorem{theorem}{Theorem}[section]
\newtheorem{proposition}[theorem]{Proposition}
\newtheorem{lemma}[theorem]{Lemma}
\newtheorem{corollary}[theorem]{Corollary}

\theoremstyle{remark}
\newtheorem{remark}[theorem]{Remark}

\begin{document}
	
\title[Bessel and Riesz Composition Formulas]{A Bessel Analog of the Riesz Composition Formula} 
	
\author[C.\ Fischbacher]{Christoph Fischbacher}
\address{Department of Mathematics, Baylor University, Sid Richardson Bldg., 1410 S.\,4th Street, Waco, TX 76706, USA}
\email{\mailto{C\_Fischbacher@baylor.edu}}
\urladdr{\url{https://math.artsandsciences.baylor.edu/person/christoph-fischbacher-phd}}

\author[F.\ Gesztesy]{Fritz Gesztesy}
\address{Department of Mathematics, Baylor University, Sid Richardson Bldg., 1410 S.\,4th Street, Waco, TX 76706, USA}
\email{\mailto{Fritz\_Gesztesy@baylor.edu}}
\urladdr{\url{https://math.artsandsciences.baylor.edu/person/fritz-gesztesy-phd}}

\author[R.\ Nichols]{Roger Nichols}
\address{Department of Mathematics (Dept.~6956), The University of Tennessee at Chattanooga, 615 McCallie Avenue, Chattanooga, TN 37403, USA}
\email{\mailto{Roger-Nichols@utc.edu}}
\urladdr{\url{https://sites.google.com/mocs.utc.edu/rogernicholshomepage/home}}

\dedicatory{Dedicated, with great admiration, to the memory of Larry Zalcman $($1943--2022\,$)$}

\date{\today}
\@namedef{subjclassname@2020}{\textup{2020} Mathematics Subject Classification}

\begin{abstract} 
We provide an elementary derivation of the Bessel analog of the celebrated Riesz composition formula and use the former to effortlessly derive the latter.
\end{abstract}

\maketitle


{\scriptsize{\tableofcontents}}

\section{Introduction} \lb{s1} 

This note aims at a pedagogical introduction into Riesz and Bessel potentials and some of their basic properties and connections with Bessel functions. In particular, we focus on offering an elementary and straightforward approach. 

Given $n \in \bbN$, the celebrated Riesz composition formula
\begin{align} \lb{1.1}
&\int_{\R^n}  d^n x^{\p} \, |x-x^{\p}|^{2\alpha-n}|x^{\p}-y|^{2\beta-n} 
= k_{\alpha,\beta,n}|x-y|^{2\alpha+2\beta-n}, \quad x, y \in \bbR^n, \; x \neq y,   \no  \\
& \hspace*{7.5cm} \alpha, \beta, (\alpha + \beta) \in (0,n/2), 
\end{align}
where
\begin{equation} \lb{1.2}
k_{\alpha,\beta,n}=\pi^{n/2}\frac{\Gamma(\alpha)\Gamma(\beta)\Gamma((n/2)-\alpha-\beta)}
{\Gamma(\alpha+\beta)\Gamma((n/2)-\alpha)\Gamma((n/2)-\beta)} 
= \f{\gamma_{\alpha,n} \gamma_{\beta,n}}{\gamma_{\alpha + \beta,n}},
\end{equation}
and 
\begin{equation} 
\gamma_{\alpha,n} = \pi^{n/2} 2^{2 \alpha} \Gamma(\alpha)/\Gamma((n/2)-\alpha),   \lb{1.4}
\end{equation}
is cited at numerous places, but rarely proved in all details. Indeed, the sources 
\cite[eq.~(11.2)]{AS56}, \cite[eq.~(1.1.3)]{La72}, \cite[Sect.~V.1]{St70}, 
all mention \eqref{1.1}, \eqref{1.2}, but only du Plessis \cite[Theorem~3.1]{Du70} provides a detailed (and not-so-short) proof. 

From the point of view of operator theory, the Riesz potential operator $\cR_{\alpha,n}$ in $L^2(\bbR^n;d^nx)$, given by  
\begin{align}
\begin{split} 
(\cR_{\alpha,n} f)(x) = \gamma_{\alpha,n}^{-1} \int_{\bbR^n} d^n y \, |x - y|^{2\alpha - n} f(y)\, \text{ for a.e.~$x\in \bbR^n$},&     \lb{1.3} \\
\alpha \in (0,n/2), \; f \in \cL(\bbR^n),&   
\end{split} 
\end{align} 
formally (and, admittedly, somewhat naively) corresponds to the $- \alpha$-th power of the Laplacian in $L^2(\bbR^n;d^nx)$, $(- \Delta)^{- \alpha}\big|_{\cL(\bbR^n)}$, $\alpha \in (0,n/2)$, with $\cL(\bbR^n)$ a Lizorkin space, see Section \ref{s3} for details. Hence, on a formal level (to be made precise later, see, e.g., \eqref{3.8}, \eqref{3.21}), the Riesz composition formula \eqref{1.1} corresponds to the semigroup property\footnote{The reader should keep in mind that what we denote by $(-\Delta)^{- \alpha}$, $\alpha \in (0,n/2)$, is often denoted by $(-\Delta)^{-\beta/2}$, $\beta \in (0,n)$.}
\begin{equation}
(- \Delta)^{- \alpha} (- \Delta)^{- \beta} = (- \Delta)^{- \alpha - \beta}, \quad \alpha, \beta, (\alpha + \beta) \in (0,n/2). 
\end{equation}

We will derive \eqref{1.1} via first principles from an analogous composition formula for Bessel potentials, namely, given $n \in \bbN$, 
\begin{align}
&\int_{\R^n} d^n x^{\p} \, \frac{\lambda^{(n-2\alpha)/4} K_{(n/2)-\alpha}\big(\lambda^{1/2}|x-x^{\p}|\big) \lambda^{(n-2\beta)/4} K_{(n/2)-\beta}\big(\lambda^{1/2}|x^{\p}-y|\big)}{|x-x^{\p}|^{(n/2)-\alpha} |x^{\p}-y|^{(n/2)-\beta}}    \no\\
&\quad=\kappa_{\alpha,\beta,n}\lambda^{(n-2\alpha-2\beta)/4}\frac{K_{(n/2)-\alpha-\beta}\big(\lambda^{1/2}|x-y|\big)}{|x-y|^{(n/2)-\alpha-\beta}},      \lb{1.7} \\
& \hspace*{2.45cm} \alpha, \beta, \lambda \in (0,\infty), \; x,y \in \R^n, \;   x \neq y,    \no
\end{align}
where
\begin{equation} \lb{1.8}
\kappa_{\alpha,\beta,n} 
=\frac{(2 \pi)^{n/2} 2^{-1} \Gamma(\alpha)\Gamma(\beta)}{\Gamma(\alpha+\beta)} = \f{\eta_{\alpha,n} \eta_{\beta,n}}{\eta_{\alpha+\beta,n}}, 
\end{equation}
and 
\begin{equation} 
\eta_{\alpha,n} = (2 \pi)^{n/2} 2^{\alpha-1} \Gamma(\alpha).  
\end{equation}
Restricting $\alpha, \beta \in (0,n/2)$ with $\alpha + \beta < n/2$, and applying the limit $\lambda\downarrow 0$ in \eqref{1.7} results in the Riesz composition formula \eqref{1.1}, \eqref{1.2}.

From an operator theory point of view, the (bounded) Bessel potential operator $\cB_{\alpha,n,z}$ in $L^2(\bbR^n;d^nx)$, given by  
\begin{align}
(\cB_{\alpha,n,z} f)(x) 
&= \eta_{\alpha,n}^{-1} (-z)^{(n-2\alpha)/4} \int_{\bbR^n} d^n y \, |x - y|^{\alpha - (n/2)}    \no \\
&\quad \times K_{(n/2)-\alpha} \big((-z)^{1/2}|x-y|\big) f(y)\, \text{ for a.e.~$x\in \bbR^n$},      \lb{1.9} \\
&\hspace*{1.2cm} \alpha \in (0,\infty), \; z \in \bbC \backslash [0,\infty), \; f \in L^2(\bbR^n;d^nx),     \no   
\end{align} 
coincides with  the $- \alpha$-th power of the Helmholtz Laplacian $- \Delta - z I_{L^2(\bbR^n;d^nx)}$ in $L^2(\bbR^n;d^nx)$,
\begin{equation}
\cB_{\alpha,n,z} = (- \Delta - z I_{L^2(\bbR^n;d^nx)})^{- \alpha}, \quad \alpha \in (0,\infty), \ z \in \bbC \backslash [0,\infty),
\end{equation}
and hence the Bessel (resp., Helmholtz) composition formula \eqref{1.7} corresponds to the semigroup property
\begin{equation}\lb{1.11}
\begin{split}
&(- \Delta - z I_{L^2(\bbR^n;d^nx)})^{- \alpha} (- \Delta - z I_{L^2(\bbR^n;d^nx)})^{- \beta}\\
&\quad = (- \Delta - z I_{L^2(\bbR^n;d^nx)})^{- \alpha - \beta},\quad \alpha, \beta \in (0,\infty),\, z\in \bbC\backslash[0,\infty).
\end{split}
\end{equation}

We conclude this introduction with a summary of the remaining sections.  In Section \ref{s2} we investigate inverse fractional powers of the Helmholtz Laplacian $- \Delta - z I_{L^2(\bbR^n;d^nx)}$ in $L^2(\bbR^n;d^nx)$.  In particular, we recall the auxiliary scalar function $g_{\alpha,n}(z;\,\cdot\,) = (|\cdot|-z)^{-\alpha}$, show that its inverse Fourier transform $g_{\alpha,n}^{\vee}(z;\,\cdot\,)$ is the convolution kernel associated with the convolution operator $(- \Delta - z I_{L^2(\bbR^n;d^nx)})^{-\alpha}$ in $L^2(\bbR^n;d^nx)$ for $\alpha\in \bbC$, $\Re(\alpha)\in (0,\infty)$, and $z\in \bbC\backslash[0,\infty)$, and discuss its $L^p$ properties.  Additional pointwise properties of $g_{\alpha,n}^{\vee}(z;\,\cdot\,)$, including a pointwise bound in terms of a Riesz-type potential that plays a key role in our derivation of \eqref{1.1} from \eqref{1.7}, are explored in Lemma \ref{l2.2}.  In Theorem \ref{2.3} we discuss the manner in which $(- \Delta +\lambda I_{L^2(\bbR^n;d^nx)})^{-\alpha}$ converges to $(-\Delta)^{-\alpha}$ as $\lambda\downarrow 0$.  In Section \ref{s3} we recall some of the basic properties of Riesz and Bessel potentials.  In addition, we use the Fourier transform to define (inverse) fractional powers of the self-adjoint realization $H_0$ of $-\Delta$ in $L^2(\bbR^n;d^nx)$ and establish connections to the operators of convolution with Riesz and Bessel potentials (see Theorems \ref{t2.1} and \ref{t3.2}).  Domain and operator core investigations of $H_0^{-\alpha}$ are considered in Proposition \ref{p1.2r} and Corollary \ref{c3.5}. Finally, in Section \ref{s4}, we use the semigroup property \eqref{1.11} to immediately deduce the Bessel composition formula \eqref{1.7}.  In turn, using the pointwise bound on $g_{\alpha,n}^{\vee}(-\lambda;\,\cdot\,)$ in terms of a Riesz-type potential from Section \ref{s2} to apply Lebesgue's dominated convergence theorem, we deduce the Riesz composition formula \eqref{1.1} by taking the limit $\lambda\downarrow 0$ throughout \eqref{1.7}. Appendix \ref{sA} discusses a useful result on a certain linear operator of convolution type. 

\medskip
\noindent 
{\bf Notation.} The symbols $f^{\wedge}(\dott)$ and $f^{\vee}(\dott)$ denote the Fourier and inverse Fourier transforms, respectively, of appropriate functions $f$; see \eqref{A.1} for the precise definitions.  The symbol $f\ast g$ denotes the convolution of an appropriate pair of functions $f$ and $g$; see \eqref{A.4}.  The symbols $J_{\nu}(\,\cdot\,)$ and $K_{\nu}(\,\cdot\,)$ denote the Bessel function and modified Bessel function of order $\nu$, respectively; see \cite[Ch.~9]{AS72}.  The inner product in a separable (complex) Hilbert space $\cH$ is denoted by $(\,\cdot\,,\,\cdot\,)_{\cH}$ and is assumed to be linear with respect to the second argument.   If $T$ is a linear operator mapping (a subspace of) a Hilbert space into another, then $\dom(T)$ denotes the domain of $T$.  The resolvent set, spectrum, and point spectrum (i.e., the set of eigenvalues) of a closed linear operator in $\cH$ will be denoted by $\rho(\,\cdot\,)$, $\sigma(\, \cdot \,)$, and $\sigma_p(\, \cdot \,)$, respectively. Similarly, the absolutely continuous and singularly continuous spectrum of a self-adjoint operator in $\cH$ are denoted by $\sigma_{ac}(\, \cdot \,)$ and $\sigma_{sc}(\, \cdot \,)$.  The Banach space of bounded linear operators on $\cH$ is denoted by $\cB(\cH)$.  Finally, for $p\in [1,\infty)$, the corresponding $\ell^p$-based trace ideals will be denoted by $\cB_p (\cH)$ with norms abbreviated by $\|\,\cdot\,\|_{\cB_p(\cH)}$.

\section{On Fractional Powers of (Helmholtz) Laplacians}  \lb{s2}

In this section we consider fractional powers of the resolvent of the self-adjoint realization of (minus) the Laplacian $-\Delta$, and the Helmholtz Laplacian $-\Delta - z I_{L^2(\bbR^n;d^nx)}$, $z\in \bbC\backslash[0,\infty)$, in $L^2(\bbR^n;d^nx)$, $n\in \bbN$.  Here the self-adjoint realization of $-\Delta$ in $L^2(\bbR^n;d^nx)$, denoted by $H_0$, is defined according to
\begin{equation}\lb{2.1}
H_0f = -\Delta f,\quad f\in \dom(H_0)=W^{2,2}(\bbR^n),
\end{equation}
where $W^{2,2}(\bbR^n) \equiv H^2(\bbR^n)$ is the Sobolev space of all functions in $L^2(\bbR^n;d^nx)$ possessing weak derivatives up to order two in $L^2(\bbR^n;d^nx)$. Since via Fourier transform, the operator $H_0$ is unitarily equivalent to the operator of multiplication by the function $|\xi|^2$, $\xi \in \bbR^n$, in $L^2(\bbR^n;d^n\xi)$, the spectrum of $H_0$ is purely absolutely continuous and given by 
\begin{equation}\lb{2.2}
\sigma(H_0) = \sigma_{ac}(H_0)=[0,\infty), \quad \sigma_{sc}(H_0) = \sigma_{p}(H_0) = \emptyset.
\end{equation}

The resolvent operator $(H_0 - zI_{L^2(\bbR^n;d^nx)})^{-1}$, $z\in \bbC\backslash [0,\infty)$, is well-known to be a convolution-type integral operator with integral kernel given by (cf., e.g., \cite[eq.~(2.19)]{GN20})
\begin{align}\lb{2.3}
\begin{split}
(H_0 - zI_{L^2(\bbR^n;d^nx)})^{-1}(x,y) = G_{1,n}(z;x,y) := g_{1,n}^{\vee}(z;x-y),&    \\
x,y\in \bbR^n,\, x\neq y,\, z\in \bbC\backslash [0,\infty),\, n\in \bbN,&
\end{split} 
\end{align}
where\footnote{We choose the branch of $(-z)^{\gamma}$, $\gamma \in \bbR \backslash\{0\}$, such that $(-z)^{\gamma} = \lambda^{\gamma} > 0$ for $z = - \lambda$, $\lambda \in (0,\infty)$.}
\begin{align}
\begin{split} 
& g_{1,n}^{\vee}(z;x)=
\begin{cases}
\big[2(-z)^{1/2}\big]^{-1}e^{- (-z)^{1/2} |x|}, \quad x \in \bbR, \; n=1,\\[1mm]
(2 \pi)^{-n/2} (-z)^{(n-2)/4} |x|^{(2-n)/2}K_{(n-2)/2}\big((-z)^{1/2} |x|\big),    \\
\hspace*{4.2cm}  x \in \bbR^{n} \backslash \{0\}, \; n\in \bbN\backslash\{1\};  
\end{cases}    \\
& \hspace*{8.03cm} z\in \bbC\backslash[0,\infty).     \lb{2.4} 
\end{split} 
\end{align}
Here our choice of notation $g_{1,n}^{\vee}$ in \eqref{2.4} anticipates the content of Lemma \ref{lA.2}.
		
Before extending \eqref{2.3}, \eqref{2.4} to fractional powers and discussing the underlying convolution operator aspects of $(H_0 -z I_{L^2(\bbR^n;d^nx)})^{-\alpha}$, $\alpha \in \bbC$, $\Re(\alpha) \in (0,\infty)$, $z \in \bbC \backslash [0,\infty)$, in $L^2(\bbR^n;d^nx)$, we now introduce the function
\begin{align}
& g_{\alpha,n}^{\vee} (z; x) =
(2 \pi)^{-n/2} 2^{1 - \alpha} \Gamma(\alpha)^{-1} (-z)^{(n - 2 \alpha)/4}|x|^{\alpha - (n/2)} 
K_{(n/2) - \alpha} \big((-z)^{1/2} |x|\big),  \no \\
& \hspace*{2cm} x \in \R^n \backslash \{0\}, \; n\in\N, \; \alpha \in \bbC, \, \Re(\alpha) \in (0,\infty), \; z \in \bbC \backslash [0,\infty).   \lb{2.5}
\end{align}

The Fourier transform of $g_{\alpha,n}^{\vee} (z;\, \cdot\,)$ may be explicitly computed as follows.  Formulas equivalent to \eqref{2.6} below may be found in \cite[(4,1) and (4,6)]{AS61} (cf.~also \cite[No.~9.6.4]{AS72} and \cite[(9.42) and Exercise 13.2 (d)]{Jo82}).

\begin{lemma} \lb{lA.2} 
Let $n \in \bbN$. Then, 
\begin{align}
\begin{split} 
& g_{\alpha,n} (z; \xi) 
= \big(g_{\alpha,n}^{\vee} (z; \, \dott \,)\big)^{\wedge}(\xi) = \big(|\xi|^2 - z\big)^{-\alpha},      \lb{2.6} \\ 
& \xi \in \bbR^n, \; \alpha \in \bbC, \, \Re(\alpha) \in (0,\infty), \; z \in \bbC \backslash [0,\infty).   
\end{split}
\end{align}
\end{lemma} 
\begin{proof}
It suffices to assume $-z = \lambda \in (0,\infty)$ and $\alpha \in (0,\infty)$ as the case $z \in \bbC \backslash [0,\infty)$ and $\Re(\alpha) \in (0,\infty)$ follows via separate analytic continuations (invoking Morera's theorem).

For now, assume $n \in \bbN \backslash \{1\}$; we will comment on the case $n=1$ later. Introducing spherical coordinates in $\bbR^n$, $n \in \bbN$, $n \geq 2$, Cartesian and polar coordinates (cf.\, e.g., \cite{Bl60}) on $\bbS^{n-1}$ are related via  
\begin{align} 
& x = (x_1,\dots,x_n) \in \bbR^n,    \no \\
& x = r \omega, \; \omega = \omega(\theta) = \omega(\theta_{1},\theta_{2},\dots,\theta_{n-1}) 
= x/|x| \in \bbS^{n-1},   \lb{2.7} \\
& x_k \in \bbR, \, 1 \leq k \leq n, \; r = |x| \in[0,\infty), \; \theta_{1}\in[0,2\pi), \; \theta_{j}\in[0,\pi], \, 2 \leq j\leq n-1,   \no 
\end{align} 
where (cf., e.g., \cite{Bl60}, \cite[Sect.~1.5]{DX13}) 
\begin{equation}\lb{2.8}
\begin{cases}
x_{1}=r \sin(\theta_1) \prod\limits_{j=2}^{n-1} \sin(\theta_j),    \\[1mm] 
x_{2}=r \cos(\theta_1) \prod\limits_{j=2}^{n-1} \sin(\theta_j),  \\
\; \vdots\\
x_{n-1}=r \cos(\theta_{n-2}) \sin(\theta_{n-1}), \\
x_{n}=r \cos(\theta_{n-1}).
\end{cases}
\end{equation}
The surface measure $d^{n-1}\omega$ on $\bbS^{n-1}$ and the volume element in $\bbR^n$ then read
\begin{equation}
d^{n-1}\omega(\theta) = d \theta_1 \prod_{j=2}^{n-1}[\sin(\theta_j)]^{j-1} d\theta_j, \quad d^{n}x=r^{n-1}dr\,d^{n-1}\omega (\theta).    \lb{2.9}
\end{equation}
With this notation at hand, one obtains, see, \cite[p.~65--66]{Bl60}, 
\begin{align}
& \int_{\bbS^{n-1}} d^{n-1}\omega  \, e^{\pm i |\xi| r \cos(\theta_{n-1})} = \f{2 \pi^{(n-1)/2}}{\Gamma((n-1)/2)}
\int_0^{\pi} [\sin(\theta_{n-1})]^{n-2} d \theta_{n-1} \, e^{\pm i |\xi| r \cos(\theta_{n-1})}    \no \\
&\quad = \f{2 \pi^{(n-1)/2}}{\Gamma((n-1)/2)} \int_0^{\pi}  [\sin(\theta_{n-1})]^{n-2} \, d \theta_{n-1}  \, \cos( |\xi| r \cos(\theta_{n-1}))    \no \\
&\quad = 2 \pi^{n/2} \bigg(\f{2}{|\xi| r}\bigg)^{(n-2)/2} J_{(n-2)/2} (|\xi| r),     \lb{2.10} 
\end{align}
where we used \cite[No.~9.1.20]{AS72}, an integral representation for the regular (at $x=0$) Bessel function $J_{\mu}(\, \dott \,)$ in the last step; namely,
\begin{align}
J_{\mu}(\zeta) &= (\zeta/2)^{\mu} \sum_{k=0}^{\infty} \f{\big(- \zeta^2/2\big)^k}{k! \Gamma(\mu + k + 1)}, 
\quad \zeta \in \bbC \backslash (-\infty,0], \; \mu \in \bbC \backslash \{ - \bbN\},    \lb{2.11} \\
&= \pi^{-1/2} \Gamma(\mu + (1/2))^{-1} (\zeta/2)^{\mu} \int_0^{\pi} [\sin(\theta)]^{2 \mu} d \theta \, 
\cos(\zeta \cos(\theta)),    \lb{2.12} \\ 
& \hspace*{4.4cm} \zeta \in \bbC \backslash (-\infty,0], \; \Re(\mu) > - 1/2.   \no 
\end{align}

Since for $\alpha, \lambda \in (0,\infty)$, $g_{\alpha,n}^{\vee}(- \lambda;\, \dott \,) \in L^1(\bbR^n)$ by \eqref{2.27}, one then computes 
\begin{align}
& \int_{\bbR^n} d^n x \, |x|^{\alpha - (n/2)} K_{(n/2) - \alpha}\big(\lambda^{1/2} |x|\big) e^{- i x \cdot \xi}    \no \\
& \quad = \int_0^{\infty} r^{n-1} dr \, r^{\alpha - (n/2)} K_{(n/2) - \alpha} \big(\lambda^{1/2} r\big) 
\int_{\bbS^{n-1}}  d^{n-1}\omega \, e^{- i |\xi| r \cos(\theta_{n-1})}     \no \\ 
& \quad = 2 \pi^{n/2} \bigg(\f{2}{|\xi|}\bigg)^{(n-2)/2} 
\int_0^{\infty} dr \, r^{\alpha} K_{(n/2) - \alpha} \big(\lambda^{1/2} r\big) J_{(n-2)/2} (|\xi| r),  \no \\
&  = 2^{(n/2) + \alpha - 1} \pi^{n/2} \lambda^{(2 \alpha - n)/4} 
\Gamma(\alpha) \big(|\xi|^2 + \lambda\big)^{- \alpha}.     \lb{2.13}
\end{align}
employing \eqref{2.10} and \cite[No.~6.5767]{GR80}.  

Finally, if $n=1$, then using that 
\begin{equation}
J_{-1/2}(|\xi|r) = [2/(\pi |\xi| r)]^{1/2} \cos(|\xi|r),\quad r\in (0,\infty),     \lb{2.20}
\end{equation}
and the fact that $|\cdot|^{\alpha-(1/2)}K_{1/2-\alpha}(\lambda^{1/2}|\cdot|)$ is an even function, this implies that the first and third integrals in \eqref{2.13} (setting $n=1$) are, in fact, equal.

Putting \eqref{2.5} and \eqref{2.13}--\eqref{2.20} together proves \eqref{2.6} for $\alpha \in (0,\infty)$. 
\end{proof}

Applying the large argument asymptotics of $K_{\nu}(\,\cdot\,)$ (see, e.g., \cite[No.~9.7.2]{AS72}),
\begin{equation}
K_{\nu}(\zeta) \underset{|\zeta| \uparrow \infty}{\sim} (\pi/2)^{1/2}\zeta^{-1/2}e^{-\zeta},\quad \nu\in \bbC,
\end{equation}
one infers that $g_{\alpha,n}^{\vee}(z;x)$ has exponential falloff as $|x| \to \infty$ of the form:
\begin{equation}\lb{2.22}
\begin{split}
g_{\alpha,n}^{\vee}(z;x) \underset{|x| \uparrow \infty}{\sim} c_{\alpha,n} |x|^{(2 \alpha - n -1)/2} (-z)^{(n-2\alpha-1)/4} e^{- (-z)^{1/2} |x|},&\\ 
\Re(\alpha) \in (0,\infty), \; z  \in \bbC \backslash [0,\infty),&
\end{split}
\end{equation}
where
\begin{equation}
c_{\alpha,n} = 2^{(1-n-2\alpha)/2}\pi^{(1-n)/2}\Gamma(\alpha)^{-1}.
\end{equation}
In addition, applying the small argument asymptotics of $K_{\nu}(\,\cdot\,)$ (see, e.g., \cite[No.~9.6.8 and No.~9.6.9]{AS72}),
\begin{align}
K_0(\zeta) &\underset{|\zeta| \downarrow 0}{\sim} -\ln(\zeta),\no\\
K_{\nu}(\zeta) &\underset{|\zeta| \downarrow 0}{\sim} 2^{\nu-1}\Gamma(\nu)\zeta^{-\nu},\quad \Re(\nu)>0,\lb{2.24} \\
K_{\nu}(\zeta) &\underset{|\zeta| \downarrow 0}{\sim} 2^{-\nu-1}\Gamma(-\nu)\zeta^{\nu},\quad \Re(\nu)<0\no
\end{align}
(note that the third asymptotic relation in \eqref{2.24} follows from the second combined with the reflection identity $K_{\nu}(\,\cdot\,) = K_{-\nu}(\,\cdot\,)$, $\nu\in \bbC$; see, \cite[No.~9.6.6]{AS72}), one infers that $g_{\alpha,n}^{\vee}(z;x)$ behaves near $x=0$ like 
\begin{align} \lb{2.25}
\begin{split} 
&	g_{\alpha,n}^{\vee}(z;x) \underset{|x| \downarrow 0}{\sim} C_{\alpha,n} \begin{cases} |x|^{2 \alpha - n}, 
& \Re(\alpha) \in (0,n/2), \\
\ln\big((-z)^{1/2}|x|\big), & \alpha = n/2, \\
\Oh(1), & \Re(\alpha) = n/2, \; \Im(\alpha) \in \bbR \backslash \{0\}, \\
(-z)^{(n-2\alpha)/2}, & \Re(\alpha) \in (n/2,\infty),
\end{cases}     \\
& \hspace*{8.45cm}	 z \in \bbC \backslash [0,\infty),
\end{split} 
\end{align}
where 
\begin{equation}
C_{\alpha,n} = \begin{cases}
\pi^{-n/2}2^{-2\alpha}\Gamma(\alpha)^{-1}\Gamma((n/2)-\alpha), & \Re(\alpha)\in (0,n/2), \\
- \pi^{-n/2} 2^{1-n} \Gamma(n/2), & \alpha = n/2, \\
\pi^{-n/2} 2^{-n} \Gamma(\alpha)^{-1}\Gamma(\alpha-(n/2)), & \Re(\alpha)\in (n/2,\infty).
\end{cases}
\end{equation}	

The case $\Re(\alpha) = n/2$, $\Im(\alpha) \in \bbR\backslash\{0\}$, in \eqref{2.25} is somewhat tricky as then the asymptotics exhibits bounded but oscillatory behavior as $|x| \downarrow 0$. This follows from combining \cite[No.~13.6.10]{{OLBC10}} (or \cite[No.~13.6.21]{AS72}), the expression of $K_{\nu}(\dott)$ in terms of the irregular confluent hypergeometric function, and the asymptotic behavior \cite[No.~13.2.18]{{OLBC10}} of the latter (see also \cite[Sect.~10.45]{{OLBC10}}). 

In particular, \eqref{2.22} and \eqref{2.25} imply:
\begin{equation}
g_{\alpha,n}^{\vee}(z;\, \dott \,) \in L^1(\bbR^n;d^nx), \quad \Re(\alpha) \in (0,\infty), \; z \in \bbC \backslash [0,\infty),  \lb{2.27} 
\end{equation}
and, using the explicit form of $g_{\alpha,n}(z;\, \dott \,)$ in \eqref{2.6},
\begin{equation}\lb{2.28} 
\begin{split} 
g_{\alpha,n}(z;\, \dott \,)  \in L^p(\bbR^n;d^n \xi), \; 
g_{\alpha,n}^{\vee}(z;\, \dott \,) \in L^1(\bbR^n;d^nx) \cap L^{q}(\bbR^n;d^nx),&  \\ 
p\in(1,\infty),\; p^{-1}+q^{-1}=1, \;\Re(\alpha) \in (n/(2p),\infty), \; z \in \bbC \backslash [0,\infty).&    
\end{split} 
\end{equation}			

Thus, Theorem \ref{tA.1}\,$(ii)$ applies to $g_{\alpha,n}(z,\xi) = \big(|\xi|^2 - z\big)^{-\alpha}$, $\xi \in \bbR^n$, $z \in \bbC \backslash [0,\infty)$, for all $\Re(\alpha) \in (0,\infty)$ (also, Theorem \ref{tA.1}\,$(i)$ applies for all $\Re(\alpha) \in (n/4,\infty)$), and one concludes that $(H_0 - z I_{L^2(\bbR^n;d^nx)})^{-\alpha}$, $\alpha \in \bbC$, $\Re(\alpha) \in (0,\infty)$, $z \in \bbC \backslash [0,\infty)$, is a convolution integral operator with integral kernel given by  
\begin{align}
(H_0 - z I_{L^2(\bbR^n;d^nx)})^{-\alpha}(x,y) = G_{\alpha,n}(z;x, y) 
= g_{\alpha,n}^{\vee}(z;x-y),&   \lb{2.29} \\
x, y \in \R^n, \;  x \neq y, \; 
\Re(\alpha) \in (0,\infty), \; z \in \bbC \backslash [0,\infty), \; n \in \bbN.&       \no
\end{align}
In particular, 
\begin{align}
& \big((H_0 - z I_{L^2(\bbR^n;d^nx)})^{-\alpha} f\big)(x) = \int_{\bbR^n} d^n y \, 
G_{\alpha,n}(z;x,y) f(y) \no \\
& \quad = \int_{\bbR^n} d^n y \, g_{\alpha,n}^{\vee}(z; x - y) f(y)      \lb{2.30}  \\
& \quad = \big(g_{\alpha,n}^{\vee}(z;\dott) * f\big)(x), \quad \alpha \in(0,\infty), \, z \in \bbC \backslash [0,\infty),  
\, f \in L^2(\bbR^n;d^nx).  \no
\end{align}

In this context we note that Young's (convolution) inequality, \eqref{A.3}, applies to \eqref{2.30} with $p=1$, $q=r=2$, $g = g_{\alpha,n}^{\vee}$, $h = f$, see \eqref{3.23}.

We continue with the following result on Bessel potentials.

\begin{lemma} \lb{l2.2}
Suppose that $n\in \bbN$ and $\alpha, \beta \in \bbC$, $\Re(\alpha), -\Re(z) \in (0,\infty)$.~Then   
the following items $(i)$--$(iii)$ hold:\\[1mm]
$(i)$ For $\alpha, \lambda \in (0,\infty)$, the Bessel potential $g_{\alpha,n}^{\vee}(-\lambda;\dott)$ is strictly monotone decreasing with respect to increasing $\lambda \in (0,\infty)$. \\[1mm]
$(ii)$ The Bessel potential is pointwise dominated by a Riesz-type potential, that is,
\begin{align}
\begin{split} 
\big|g_{\alpha,n}^{\vee}(z;x)\big| < 2^{n/2} 
\f{\Gamma((n/2)-\Re(\alpha))}{2^{2 \Re(\alpha)} |\Gamma(\alpha)|} |x|^{2 \Re(\alpha)-n},&  \\
\Re(\alpha) \in (0,n/2), \; -\Re(z) \in (0,\infty), \; x \in \bbR^n \backslash \{0\}.&   \lb{2.31} 
\end{split} 
\end{align} 
In particular, in the special real-valued case one has
\begin{align}
\begin{split}
& g_{\alpha,n}^{\vee}(-\lambda;x) < \gamma_{\alpha,n}^{-1} |x|^{2\alpha - n},   \\
& \alpha \in (0,n/2), \; \lambda \in (0,\infty), \; x \in \bbR^n \backslash \{0\}. 
\end{split} 
\end{align}
$(iii)$ The Bessel potential converges to the Riesz potential as $z \to 0$, $z \in S_{\varepsilon}$
\begin{equation}
\lim_{\substack{z \to 0 \\ z \in S_{\varepsilon}}} g_{\alpha,n}^{\vee}(z;x) 
= \gamma_{\alpha,n}^{-1} |x|^{2\alpha-n}, \quad \Re(\alpha) \in (0,n/2), \; \; x \in \bbR^n \backslash \{0\}, 
\end{equation}
where $S_{\varepsilon}$ denotes the sector
\begin{equation}
S_{\varepsilon} = \{z \in \bbC\backslash\{0\}\,|\, \arg(z) \in [(\pi/2) + \varepsilon, (3\pi/2) - \varepsilon]\},
\end{equation}
for some $\varepsilon \in (0, \pi/2]$. 
\end{lemma}
\begin{proof}
The integral representation 
\begin{equation}
2 (\beta/\gamma)^{\nu/2} K_{\nu}\big(2 (\beta \gamma)^{1/2}\big) 
= \int_0^{\infty} dt \, e^{- (\beta/t) -\gamma t} t^{\nu -1}, 
\quad \Re(\beta), \Re(\gamma) \in (0,\infty), \; \nu \in \bbR, 
\lb{2.35} 
\end{equation}
see, for instance, \cite[No.~3.4719]{GR80}, \cite[p.~85]{MOS66} (see also \cite[Subsect.~1.2.4]{AH96}, \cite[p.~233]{Ha06}, \cite[p.~296]{MS01}, \cite[p.~132]{St70}), yields
\begin{align}
& g_{\alpha,n}^{\vee}(z;x)   \no \\
& \quad = 2^{1-\alpha} \Gamma(\alpha)^{-1} 
(-z)^{(n - 2\alpha)/4} |x|^{\alpha -(n/2)} K_{(n/2)-\alpha}\big((-z)^{1/2} |x|\big)   \no \\
& \quad = 2^{-n/2} \Gamma(\alpha)^{-1} \int_0^{\infty} dt \, e^{z/t} e^{-|x|^2 t/4}t^{(n/2) - \alpha -1},    \lb{2.36} \\
& \quad = 2^{-n/2} \Gamma(\alpha)^{-1} \int_0^{\infty} ds \, e^{zs} e^{-|x|^2/(4s)}s^{\alpha - (n/2) - 1},    \lb{2.37} \\
& \hspace*{2.55cm} \Re(\alpha), -\Re(z) \in (0,\infty), \; x \in \bbR^n \backslash \{0\}.   \no  
\end{align} 
Thus, relation \eqref{2.36} implies strict monotone decreasing of $g_{\alpha,n}^{\vee}(-\lambda;\dott)$, $\alpha \in (0,\infty)$, with respect to increasing $\lambda \in (0,\infty)$, and hence proves item $(i)$. 

Employing $|e^{z/t}| < 1$, $-\Re(z), t \in (0,\infty)$, in \eqref{2.36} also yields the estimate 
\begin{align}
& \big|g_{\alpha,n}^{\vee}(z;x)\big|   \no \\
& \quad \leq 2^{-n/2} |\Gamma(\alpha)|^{-1} \int_0^{\infty} dt \, e^{\Re(z)/t} e^{-|x|^2 t/4}t^{[(n-2)/2 - \Re(\alpha)]}     \no \\
& \quad < 2^{-n/2} |\Gamma(\alpha)|^{-1} \int_0^{\infty} dt \, e^{-|x|^2 t/4}t^{[(n-2)/2 - \Re(\alpha)]}    \no \\
& \quad = 2^{n/2} \f{\Gamma((n/2)-\Re(\alpha))}{2^{2 \Re(\alpha)} |\Gamma(\alpha)|} |x|^{2 \Re(\alpha)-n},      \\
& \hspace*{1.3cm} \Re(\alpha) \in (0,n/2), \; -\Re(z) \in (0,\infty), \; x \in \bbR^n \backslash \{0\},   \no 
\end{align}
where we changed variables $t \mapsto u = t |x|^2/4$ and then employed the integral representation \cite[No.~6.1.1]{AS72} for the Gamma function in the final step, proving item $(ii)$. 

To prove item $(iii)$, one uses the estimate $\big|e^{z/t} t^{-\alpha}\big| \leq t^{-\Re(\alpha)}$ and Lebesgue's dominated convergence theorem to interchange the limit $z \to 0$, $z \in S_{\varepsilon}$, with the parameter integral in \eqref{2.36}. 
\end{proof}

We conclude this section by proving that $(H_0 + \lambda I_{L^2(\bbR^n;d^nx)})^{- \alpha}$ converges to $H_0^{- \alpha}$ in norm resolvent sense as $\lambda \downarrow 0$.  The proof of this result relies on the following explicit characterization of the norm of a bounded multiplication operator.

\begin{lemma}\lb{l2.3m}
Let $n\in \bbN$ and $h\in L^{\infty}(\bbR^n;d^nx)$.  The operator of multiplication $\cM_h:L^2(\bbR^n;d^nx)\to L^2(\bbR^n;d^nx)$ defined by $\cM_{h}:f\mapsto hf$, $f\in L^2(\bbR^n;d^nx)$ is a bounded linear operator and
\begin{equation}\lb{2.33m}
\|\cM_h\|_{\cB(L^2(\bbR^n;d^nx))} = \|h\|_{L^{\infty}(\bbR^n;d^nx)}.
\end{equation}
\end{lemma}
\begin{proof}
Let $h\in L^{\infty}(\bbR^n;d^nx)$ and $\cM_h$ the operator of multiplication by $h$ in the Hilbert space $L^2(\bbR^n;d^nx)$.  It is clear that $\cM_h$ is a linear operator and that
\begin{align}
\|\cM_h f\|_{L^2(\bbR^n;d^nx)} \leq \|h\|_{L^{\infty}(\bbR^n;d^nx)}\|f\|_{L^2(\bbR^n;d^nx)},\quad f\in L^2(\bbR^n;d^nx).
\end{align}
Hence, $\|\cM_h\|_{\cB(L^2(\bbR^n;d^nx))} \leq \|h\|_{L^{\infty}(\bbR^n;d^nx)}$.  Therefore, to establish \eqref{2.33m}, it suffices to show that
\begin{equation}\lb{2.35m}
 \|h\|_{L^{\infty}(\bbR^n;d^nx)}\leq \|\cM_h\|_{\cB(L^2(\bbR^n;d^nx))}.
\end{equation}
To this end, define the sets
\begin{equation}
A_{\varepsilon}=\{x\in \bbR^n\,|\, |h(x)| \geq \|h\|_{L^{\infty}(\bbR^n;d^nx)}-\varepsilon\},\quad \varepsilon>0.
\end{equation}
It follows from the definition of the essential supremum that
\begin{equation}
\text{$A_{\varepsilon}$ has positive Lebesgue measure for every $\varepsilon>0$}.
\end{equation}
For each $\varepsilon>0$, choose a subset $B_{\varepsilon}\subseteq A_{\varepsilon}$ such that $B_{\varepsilon}$ has positive and finite Lebesgue measure.  Then $\chi_{B_{\varepsilon}}\in L^2((\bbR^n;d^nx))$, $\varepsilon>0$, and
\begin{align}
\|\cM_h \chi_{B_{\varepsilon}}\|_{L^2(\bbR^n;d^nx)}^2 &= \int_{B_{\varepsilon}} d^nx\, |h(x)|^2\no\\
&\geq \big[\|h\|_{L^{\infty}(\bbR^n;d^nx)}-\varepsilon\big]^2 \|\chi_{B_{\varepsilon}}\|_{L^2(\bbR^n;d^nx)}^2,\quad \varepsilon>0.\lb{2.39m}
\end{align}
As a result, $\|\cM_h\|_{\cB(L^2(\bbR^n;d^nx))} \geq \|h\|_{L^{\infty}(\bbR^n;d^nx)}-\varepsilon$ for all $\varepsilon>0$, and \eqref{2.35m} follows.
\end{proof}

If $f\in L^{\infty}(\bbR^n;d^nx)$, then $f(H_0)\in \cB(L^2(\bbR^n;d^nx))$ is given via the Fourier transform by
\begin{equation}\lb{2.39n}
f(H_0) = \cF \cM_{f(|\cdot|^2)} \cF^{-1},
\end{equation}
where $\cM_{f(|\cdot|^2)}$ is the operator of multiplication by $f(|\cdot|^2)\in L^{\infty}(\bbR^n;d^n\xi)$.  (An alternative, equivalent, means of expressing $f(H_0)$ is furnished by the spectral theorem and the functional calculus for self-adjoint operators; see, e.g., \cite[Sect.~VI.5.2]{Ka80} and \cite[Sect.~7.3]{We80}.  The approach via the spectral theorem has the advantage that it applies to arbitrary self-adjoint operators, but in the case of $H_0$, the characterization via the Fourier transform in \eqref{2.39n} is more direct.)

\begin{theorem} \lb{t2.4r}
Let $\alpha \in \bbC$ with $\Re(\alpha)\in (0,\infty)$ and suppose $\lambda \in (0,1]$. Then $(H_0 + \lambda I_{L^2(\bbR^n;d^nx)})^{- \alpha}$ converges in norm resolvent sense to $H_0^{- \alpha}$ as $\lambda \downarrow 0$; that is,
\begin{align}
\begin{split}
& \lim_{\lambda \downarrow 0} \Big\|\big((H_0 + \lambda I_{L^2(\bbR^n;d^nx)})^{- \alpha} - z I_{L^2(\bbR^n;d^nx)}\big)^{-1}    \lb{2.39} \\
& \hspace*{8mm} - \big(H_0^{- \alpha} - z I_{L^2(\bbR^n;d^nx)}\big)^{-1}\Big\|_{\cB(L^2(\bbR^n;d^nx))} = 0, \quad z \in \bbC \backslash S_{\alpha},
\end{split}
\end{align}
where
\begin{align}
S_{\alpha}&=\overline{\bigcup_{\lambda\in[0,1]}\sigma\big((H_0+\lambda I_{L^2(\bbR^n;d^nx)})^{-\alpha}\big)}\no\\
&=\big\{(\mu+\lambda)^{-\alpha} \in \bbC \,\big|\,\mu\in[0,\infty),\,\lambda\in(0,1]\big\}\cup\{0\}\no\\
&=\big\{\mu^{-\alpha} \in \bbC \,\big|\,\mu\in(0,\infty)\big\}\cup\{0\}.
\end{align}
\end{theorem}
\begin{proof}
Let $\alpha \in \bbC$ with $\Re(\alpha)\in (0,\infty)$.  The equation $\mu^{-\alpha}=-1$ has no solutions for $\mu\in (0,\infty)$.  Therefore, $-1\in \bbC\backslash S_{\alpha}$.  Recalling the identity (see, e.g., \cite[p.~178]{We80}),
\begin{align}
\begin{split}
& (T_1 - z I_{\cH})^{-1} - (T_2 - z I_{\cH})^{-1} = (T_1 - z_0 I_{\cH}) (T_1 - z I_{\cH})^{-1}  \\
& \quad \times \big[(T_1 - z_0 I_{\cH})^{-1} - (T_2 - z_0 I_{\cH})^{-1}\big](T_2 - z_0 I_{\cH}) (T_2 - z I_{\cH})^{-1},    \\
& \hspace*{7.05cm}  z \in \rho(T_1) \cap \rho(T_2),
\end{split}
\end{align}
where $T_j$, $j=1,2$, are closed operators in a complex, separable Hilbert space $\cH$ with $z_0 \in \rho(T_1) \cap \rho(T_2)$, one may exploit boundedness of the factors $(T_j - z_0 I_{\cH}) (T_j - z I_{\cH})^{-1}= I_{\cH} +(z-z_0) (T_j - z I_{\cH})^{-1}$, $j=1,2$, to conclude that it suffices to take $z=-1$ in \eqref{2.39} without loss of generality. Thus, utilizing unitarity of the Fourier transform and Lemma \ref{l2.3m}, one obtains:
\begin{align}
& \Big\|\big((H_0 + \lambda I_{L^2(\bbR^n;d^nx)})^{- \alpha} + I_{L^2(\bbR^n;d^nx)}\big)^{-1}     \no \\
& \qquad \; \; - \big(H_0^{- \alpha} + I_{L^2(\bbR^n;d^nx)}\big)^{-1}\Big\|_{\cB(L^2(\bbR^n;d^nx))}     \no \\
& \quad =  \Big\|\cF \Big[\big((H_0 + \lambda I_{L^2(\bbR^n;d^nx)})^{- \alpha} + I_{L^2(\bbR^n;d^nx)}\big)^{-1}     \no \\
& \hspace*{1cm} - \big(H_0^{- \alpha} + I_{L^2(\bbR^n;d^nx)}\big)^{-1}\Big] \cF^{-1}\Big\|_{\cB(L^2(\bbR^n;d^nx))}     \no \\
& \quad = \underset{\mu \in (0,\infty)}{\rm ess.sup} \, \Big|\big((\mu + \lambda)^{-\alpha} + 1\big)^{-1} - \big(\mu^{-\alpha} + 1\big)^{-1}\Big|   \no \\
& \quad =  \underset{\mu \in [0,\infty)}{\rm ess.sup} \, |f_{\alpha}(\mu + \lambda) - f_{\alpha}(\mu)|, \quad \lambda \in (0,1], \lb{2.42}
\end{align}
where
\begin{equation}
f_{\alpha}(\mu) = \begin{cases} \big(\mu^{-\alpha} + 1\big)^{-1}, & \mu \in (0,\infty), \\
0, & \mu=0.
\end{cases}
\end{equation}
The function $f_{\alpha}$ is continuous on $[0,\infty)$ since $\mu^{-\alpha}+1\neq 0$ for all $\mu\in (0,\infty)$ and $\lim_{\mu\downarrow 0}f_{\alpha}(\mu)=0=f_{\alpha}(0)$.  In fact, since $\lim_{\mu\to \infty}f_{\alpha}(\mu)=1$, the function $f_{\alpha}$ is actually uniformly continuous on $[0,\infty)$.  In particular,
\begin{quote}
\textit{For every $\varepsilon\in (0,\infty)$, there exists $\delta(\varepsilon)\in(0,1/2)$ such that
\begin{equation}\lb{2.49}
\underset{\mu \in [0,\infty)}{\rm ess.sup}\, |f_{\alpha}(\mu+\lambda)-f_{\alpha}(\mu)|<\varepsilon,\quad \lambda\in(0,\delta(\varepsilon)).
\end{equation}
}
\end{quote}
In consequence,
\begin{equation}\lb{2.50}
\lim_{\lambda\downarrow 0}\Bigg[ \underset{\mu \in [0,\infty)}{\rm ess.sup} \, |f_{\alpha}(\mu + \lambda) - f_{\alpha}(\mu)|\Bigg] = 0,
\end{equation}
and hence \eqref{2.39} with $z=-1$ follows from \eqref{2.42} and \eqref{2.50}.
\end{proof}
\section{Some Properties of Riesz and Bessel Potentials}  \lb{s3}

We  start by recalling some basic facts on $L^p$-properties of Riesz and Bessel potentials (see, e.g., \cite[Theorem 5.9 and Corollary 5.10]{LL01}, \cite[Sect.~4.15]{Ru96}, 
\cite[Sects.~7.1--7.3]{Sa02}, and \cite[Sects.~V.1, V.3]{St70}):
\begin{theorem} \lb{t2.1}
Let $n \in \bbN$, $\alpha \in \bbC$, $\Re(\alpha) \in (0,n/2)$, $x \in \bbR^n$, and introduce the Riesz potential operator $\cR_{\alpha,n}$ as follows: 
\begin{align}
(\cR_{\alpha,n} f)(x) 
& = \gamma_{\alpha,n}^{-1} \int_{\bbR^n} d^n y \, |x - y|^{2\alpha - n} f(y)    \no \\
& = \big(g_{\alpha,n}^{\vee}(0,\dott) * f\big)(x),      \lb{3.1}  \\
& \hspace*{-1.65cm} \gamma_{\alpha,n} = \pi^{n/2} 2^{2\alpha} \Gamma(\alpha)/\Gamma((n/2)-\alpha),     \no
\end{align}
for appropriate functions $f$ $($see below\,$)$. \\[1mm] 
$(i)$ Let $p \in [1,\infty)$ and $f \in L^p(\bbR^n;d^nx)$. Then the integral $(\cR_{\alpha,n} f)(x)$ converges for $($Lebesgue\,$)$ a.e.~$x \in \bbR^n$.   \\[1mm] 
$(ii)$  Let $q\in (1,\infty)$ and $p=p(\alpha,n,q):=nq/[2+2q\Re(\alpha)]$.  There exists a constant $C_{\alpha,n,q,p}\in (0,\infty)$ such that
\begin{equation}
\|\cR_{\alpha,n} f\|_{L^q(\bbR^n;d^nx)} 
\leq C_{\alpha,n,q,p} \|f\|_{L^p(\bbR^n;d^nx)},\quad f\in L^p(\bbR^n;d^nx);  \lb{3.2} 
\end{equation}
in particular, 
\begin{equation}
\cR_{\alpha,n} \in \cB\big(L^p(\bbR^n; d^nx), L^q(\bbR^n; d^nx)\big). 
\end{equation}
$(iii)$  If $\alpha\in (0,n/2)$ and $f\in C_0^{\infty}(\bbR^n)$, then
\begin{align}
\big[|\dott|^{-2\alpha}f^{\wedge}\big]^{\vee}(x) &= \gamma_{\alpha,n}^{-1} \int_{\bbR^n}d^ny\, |x-y|^{2\alpha - n}f(y)\no\\
&= (\cR_{\alpha,n} f)(x),\quad x\in \bbR^n.\lb{3.4}
\end{align}
Moreover, the right-hand side in \eqref{3.4} is a $C^{\infty}(\bbR^n)$ function that decays as $|x|^{2\alpha-n}$ as $|x|\to \infty$.\\[1mm]
$(iv)$  If $\alpha \in (0,n/4)$ and $f\in L^{p(\alpha,n,2)}(\bbR^n;d^nx)$ with $p(\alpha,n,2) := 2n/(n+4\alpha)$, then
\begin{equation}\lb{3.5}
h_f:= \gamma_{\alpha,n}^{-1} |\dott|^{2\alpha-n}\ast f\in L^2(\bbR^n;d^nx)
\end{equation}
and
\begin{equation}\lb{3.6}
|\dott|^{-2\alpha}f^{\wedge} = (h_f)^{\wedge}.
\end{equation}
In particular, if $f\in L^{p(\alpha,n,2)}(\bbR^n;d^nx)$, then
\begin{equation}\lb{3.7}
\big[|\dott|^{-2\alpha}f^{\wedge}\big]^{\vee} = \gamma_{\alpha,n}^{-1} |\dott|^{2\alpha-n}\ast f 
= \cR_{\alpha,n} f \in L^2(\bbR^n;d^nx).
\end{equation}
\end{theorem}

In this context one may {\it naively expect} that $\cR_{\alpha,n}$, $n \in \bbN, \; \alpha \in \bbC, \, \Re(\alpha) \in (0,n/2)$, corresponds to $H_0^{- \alpha}$, but, as will be shown next, the actual details are more intriguing:

The operator $H_0^{-\alpha}$, $\alpha\in \bbR$, is most naturally characterized using the Fourier transform (under which the self-adjoint realization $H_0$ of $-\Delta$, with $\dom(H_0) = W^{2,2}(\bbR^n)$, in $L^2(\bbR^n;d^nx)$, is unitarily equivalent to the maximally defined operator of multiplication by $|\dott|^2$):
\begin{align}\lb{3.8}
&H_0^{-\alpha}f = \big[|\dott|^{-2\alpha}f^{\wedge}\big]^{\vee},\\
&f\in \dom\big(H_0^{-\alpha}\big) = \big\{g\in L^2(\bbR^n;d^nx)\,\big|\, |\dott|^{-2\alpha}g^{\wedge}\in L^2(\bbR^n;d^nx)\big\},\quad \alpha\in \bbR.\no
\end{align}
See also \cite[Sect.~2.2]{CLM20} in this context. 
Items $(iii)$ and $(iv)$ in Theorem \ref{t2.1} illustrate the action of $H_0^{-\alpha}$ as defined by \eqref{3.8} on certain elements in its domain for $\alpha\in(0,n/4)$.  In particular, Theorem \ref{t2.1} $(iv)$ confirms that
\begin{equation}
L^{p(\alpha,n,2)}(\bbR^n;d^nx)\cap L^2(\bbR^n;d^nx) \subseteq \dom\big(H_0^{-\alpha}\big),\quad \alpha \in (0,n/4),
\end{equation}
and that for $f\in L^{p(\alpha,n,2)}(\bbR^n;d^nx)\cap L^2(\bbR^n;d^nx)$,
\begin{align}
\begin{split}
\big(H_0^{-\alpha}f\big)(x) &= \gamma_{\alpha,n}^{-1} \int_{\bbR^n}d^ny\, |x-y|^{2\alpha-n}f(y)\lb{3.10}\\
&= (\cR_{\alpha,n} f)(x) \, \text{ for a.e.~$x\in \bbR^n$},\, \alpha\in (0,n/4).
\end{split} 
\end{align}
In particular, \eqref{3.10} holds if $f\in C_0^{\infty}(\bbR^n)$ when $\alpha\in(0,n/4)$.

\begin{proposition}\lb{p1.2r}
Let $n\in \bbN$.  Then $C_0^{\infty}(\bbR^n)\subset\dom\big(H_0^{-\alpha}\big)$ if and only if $\alpha\in (0,n/4)$. In addition, the analogous statement with the space $C_0^{\infty}(\bbR^n)$ replaced by the Schwartz space $\cS(\bbR^n)$ holds.
\end{proposition}
\begin{proof}
If $f\in C_0^{\infty}(\bbR^n)$, or, more generally $f \in \cS(\bbR^n)$, then $f^{\wedge}\in \cS(\bbR^n)$, where $\cS(\bbR^n)$ denotes the Schwartz space of rapidly decreasing smooth functions on $\bbR^n$.  Hence, $|\dott|^{-2\alpha}f^{\wedge} \in L^2(\bbR^n; d^nx)$ if $\alpha\in (0,n/4)$.

If $\alpha\in [n/4,\infty)$, then for any nonnegative $f\in C_0^{\infty}(\bbR^n)$, one infers that
\begin{equation}
\|f^{\wedge}\|_{L^{\infty}(\bbR^n)} = f^{\wedge}(0).
\end{equation}
Therefore, $|\dott|^{-2\alpha}f^{\wedge}$ behaves like $\|f^{\wedge}\|_{L^{\infty}(\bbR^n)}|x|^{-2\alpha}$ as $|x|\to 0$, and the latter belongs to $L^2(\bbR^n; d^nx)$ if and only if $\|f^{\wedge}\|_{L^{\infty}(\bbR^n)}=0$; that is, if and only if $f\equiv 0$.  As a result, no sign-definite function in $C_0^{\infty}(\bbR^n)\backslash\{0\}$ belongs to $\dom\big(H_0^{-\alpha}\big)$.
\end{proof}

In this context we also mention that $r_{\alpha}(x) := \gamma_{\alpha,n}^{-1} |x|^{2 \alpha - n}$, $\Re(\alpha) \in (0,n/2)$, $x \in \bbR^n \backslash \{0\}$, satisfies $r_{\alpha} \in \cS^{\prime}(\bbR^n)$ and $r_{\alpha}^{\wedge}\big|_{\bbR^n \backslash \{0\}} \in C^{\infty}(\bbR^n \backslash \{0\})$, where 
\begin{equation}\lb{3.12}
r_{\alpha}^{\wedge}(\xi) = (2\pi)^{-n/2} |\xi|^{- 2\alpha}, \quad \xi \in \bbR^n \backslash \{0\}, 
\end{equation}
see, for instance, \cite[p.~363]{GS64}, \cite[Proposition~4.64]{Mi18}, and 
\cite[Lemmas~2.13 and 2.15]{Sa02}. For analogous facts on the Bessel potential see \cite[Sect.~7.2]{Sa02}. \\[1mm] 
	
Given the result of Proposition \ref{p1.2r}, it is entirely natural to ask the question:

\begin{quote}
\textit{For which values of $\alpha$ is $C_0^{\infty}(\bbR^n)$ an operator core for $H_0^{-\alpha}$?}
\end{quote}
The following sequence of results, Lemma \ref{l3.3} and Corollaries \ref{c3.4} and \ref{c3.5}, completely answers this question, but, as pointed out by the referee, these results boil down to the fact that $C_0^{\infty}(\bbR^n)$ is dense in fractional Sobolev spaces; a thorough discussion of these results can be found, for instance, in \cite[Ch.~7]{Sa02}, in particular, see Lemma~7.15 and Theorem~7.38 therein. For completeness we include the corresponding short proofs.

\begin{lemma} \label{l3.3}
Let $\eta: \bbR^n\rightarrow\C$ be a measurable function and let $M_\eta$ be the maximal operator of multiplication by $\eta$ given by 
\begin{align}
\begin{split} 
& (M_{\eta} f)(\xi) = \eta(\xi) f(\xi) \, \text{ for a.e.~$\xi \in \bbR^n$},     \\	
&f\in \dom(M_\eta)=\big\{g\in L^2(\bbR^n;d^n \xi)\,\big|\, \eta g\in L^2(\bbR^n;d^n \xi)\big\}. 
\end{split} 
\end{align}
Then, any dense subset $\cC$ of $\dom(M_\eta)$ is a core for $M_\eta$. 
\end{lemma}
\begin{proof}
First, note that $M_\eta^*=M_{\ol{\eta}}$ and $M_{\ol{\eta}}^*=M_\eta$. Now, define $M_{{\eta},0}:=M_{{\eta}}\upharpoonright_\cC$. We claim that $M_{\eta,0}^*=M_{\ol{\eta}}$. Since $M_{\eta,0} \subseteq M_\eta$, it immediately follows that $M_{\ol{\eta}} \subseteq M_{\eta,0}^*$. Now, suppose that $f\in\cD(M_{\eta,0}^*)$, which means that there exists $\tilde{f}\in L^2(\bbR^n;d^n \xi)$ such that for all $g\in\cC$:
\begin{equation}
\begin{split}
&( f,M_{\eta,0}g)_{L^2(\bbR^n;d^nx)}  = \big(\tilde{f},g\big)_{L^2(\bbR^n;d^n \xi)}\\[1mm] 
&\quad \text{if and only if } \, \int_{\bbR^n}d^n \xi\, \ol{\left(f(\xi)\ol{\eta(\xi)}-\tilde{f}(\xi) \right)}g(\xi)=0.
\end{split}
\end{equation}
Since $\cC$ is dense, this implies $\ol{\eta} f\in L^2(\bbR^n;d^n \xi)$. Consequently, $f\in\cD(M_{\ol{\eta}})$ and $M_{\eta,0}^*f=M_{\ol{\eta}}f$. This shows $M_{\eta,0}^* \subseteq M_{\ol{\eta}}$ and therefore $M_{\eta,0}^*=M_{\ol{\eta}}$. Similarly, it can be shown that $M_{\ol{\eta},0}^*=M_\eta$. Therefore, one obtains $\ol{M_{\eta,0}}=M_{\eta,0}^{**}=M_\eta$, which proves that $\cC$ is a core for $M_\eta$.		
\end{proof}

Recalling the function $r_{\alpha}^{\wedge}$ in \eqref{3.12} one obtains the following result: 

\begin{corollary}\lb{c3.4}
For any $\alpha\in\C$, $C_0^\infty(\bbR^n \backslash \{0\})$ is a core for $M_{r_{\alpha}^{\wedge}}$.	
\end{corollary}
\begin{proof}
This follows immediately from Lemma \ref{l3.3} using fact that $C_0^\infty(\bbR^n \backslash \{0\})$ is dense in $L^2(\bbR^n;d^n \xi)$ and a subset of $\cD(M_{r_{\alpha}^{\wedge}})$ for any $\alpha\in\C$.
\end{proof}

\begin{corollary}\lb{c3.5}
If $\alpha\in \C$ with $\Re(\alpha)\in(0,n/4)$, then $C_0^\infty(\bbR^n)$ is a core for $H_0^{-\alpha}$.
\end{corollary}
\begin{proof}
By taking Fourier transforms, one infers that $C_0^\infty(\bbR^n)$ is a core for $H_0^{-\alpha}$ if and only if $\cF C_0^\infty(\bbR^n) = \big\{g^{\wedge} \in \cS(\bbR^n)\,\big|\, g\in C_0^\infty(\bbR^n)\big\}$ is a core for $M_{r_{\alpha}^{\wedge}}$. Since 
\begin{equation} 			
\cF C_0^\infty(\bbR^n)\subset\cS(\bbR^n)\subset \dom(M_{r_{\alpha}^{\wedge}}), 
\end{equation} 
the conclusion of Corollary \ref{c3.5} will follow from Lemma \ref{l3.3} upon showing that $\cF C_0^\infty(\bbR^n)$ is dense in $L^2(\bbR^n;d^n \xi)$. This can be shown as follows: if $f\in \cF C_0^\infty(\bbR^n)^\perp$, then, since 
\begin{equation} 
0= \big(f,g^\wedge\big)_{L^2(\bbR^n;d^n \xi)} = \big( f^\vee,g\big)_{L^2(\bbR^n;d^n \xi)} \, \text{ for every $g\in C_0^\infty(\bbR^n)$}, 
\end{equation} 
it follows that $f^\vee=0$, implying $f=0$.
\end{proof}

Typically, when dealing with $H_0^{-\alpha}$ and $\cR_{\alpha,n}$, one also involves the Lizorkin space 
\begin{equation}
\cL(\bbR^n) = \bigg\{f \in \cS(\bbR^n) \, \bigg| \, \int_{\bbR^n} d^n x \, x^{m} f(x) = 0 \, \text{for all $m \in \bbN_0^n$}\bigg\},
\end{equation}
and its Fourier transform,
\begin{align}
\begin{split} 
\cL^{\wedge}(\bbR^n) &= \big\{f^{\wedge} \in \cS(\bbR^n) \, \big| \, \big(\partial^m f^{\wedge}\big)(0) = 0\, \text{for all $m \in \bbN_0^n$}\big\},    \\
&= \cF \cL(\bbR^n) \supset C_0^{\infty}(\bbR^n \backslash \{0\}),
\end{split} 
\end{align}
employing the fact,
\begin{equation}
\int_{\bbR^n} d^n x \, x^m f(x) = i^{-|m|} \int_{\bbR^n} d^n x \, (i x)^m f(x) e^{i x \cdot 0} = i^{-|m|} \big(\partial^m f^{\wedge}\big)(0), \quad f \in \cS(\bbR^n).
\end{equation}
In particular,
\begin{equation}
f_0^{\wedge} (\xi) = e^{- |\xi|^2 - |\xi|^{-2}}, \; \xi \in \bbR^n, \, \text{ satisfies } \, f_0^{\wedge} \in \cL^{\wedge}(\bbR^n). 
\end{equation}
Here we use the multi-index notation for $m = (m_1,\dots,m_n) \in \bbN_0^n$,
\begin{align}
\begin{split}
& x^m = x_1^{m_1} \cdots x_n^{m_n}, \quad x =(x_1,\dots,x_n) \in \bbR^n,    \\
& |m| = m_1 + \cdots + m_n,      \\
& \partial^m = \partial_1^{m_1} \cdots \partial_n^{m_n}, \quad \partial_j = \partial/\partial x_j, \; 1\leq j \leq n.
\end{split}
\end{align}

Then $|\dott|^{-2\alpha}$ leaves $\cL^{\wedge}(\bbR^n)$ invariant and one confirms
\begin{equation}\lb{3.21}
\cR_{\alpha,n} \cR_{\beta,n} f = \cR_{\alpha+\beta,n} f, \quad f \in \cL(\bbR^n) = \cF^{-1} \cL^{\wedge}(\bbR^n);
\end{equation}
in particular, \eqref{3.21} holds for $f \in \cF^{-1} C_0^{\infty}(\bbR^n \backslash \{0\})$. For additional details in this context, see, for instance \cite[Ch.~2]{Sa02}, \cite[\S~25]{SKM93}.

While we are exclusively focused on negative powers of the Laplacian, that is, $H_0^{-\alpha}$, $\Re(\alpha) \in (0,n/2)$, we refer, for instance, to \cite{CLM20}, \cite{Kw17}, and \cite[Ch.~3]{Sa02} for an overview regarding positive powers of the Laplacian, $H_0^{\beta}$, $\beta \in (0,1)$ (and more generally, for complex $\beta$). 

In sharp contrast to the case of the Riesz potential operator, the analogous considerations for the Bessel potential operator are much simpler since $(H_0 + \lambda I_{L^2(\bbR^n;d^nx)})^{-\alpha} \in \cB\big(L^2(\bbR^n;d^nx)\big)$, $\Re(\alpha) \in (0,\infty)$. In particular, $H_0 + \lambda I_{L^2(\bbR^n;d^nx)}$ is an elliptic operator of positive-type as detailed in \cite[\S~16]{KZPS76}, and hence its fractional powers are well-understood beyond the use of the spectral theorem:

\begin{theorem} \lb{t3.2}
Let $n \in \bbN$, $\alpha \in \bbC$, $\Re(\alpha) \in (0,\infty)$, $z \in \bbC \backslash [0,\infty)$, $x \in \bbR^n$, and introduce the Bessel potential operator $\cB_{\alpha,n,z}$ as follows: 
\begin{align}
(\cB_{\alpha,n,z} f)(x) 
& = \eta_{\alpha,n}^{-1} (-z)^{(n-2\alpha)/4} \int_{\bbR^n} d^n y \, |x - y|^{\alpha - (n/2)}    \no \\
& \quad \times K_{(n/2)-\alpha} \big((-z)^{1/2}|x-y|\big) f(y)     \no \\ 
& = \big(g_{\alpha,n}^{\vee}(z,\dott) * f\big)(x),      \lb{3.22} \\
& \hspace*{-1.8cm}  \eta_{\alpha,n} = (2 \pi)^{n/2} 2^{\alpha-1} \Gamma(\alpha),      \no   
\end{align} 
for appropriate functions $f$. In addition, assuming $p \in [1,\infty) \cup \{\infty\}$ and 
$f \in L^p(\bbR^n;d^nx)$, the integral $(\cB_{\alpha,n,z} f)(x)$ converges for $($Lebesgue\,$)$ 
a.e.~$x \in \bbR^n$ and 
\begin{equation}
\|\cB_{\alpha,n,z} f\|_{L^p(\bbR^n;d^nx)} \leq  
\big\|g_{\alpha,n}^{\vee}(z,\dott)\big\|_{L^1(\bbR^n;d^nx)} \|f\|_{L^p(\bbR^n;d^nx)}.    \lb{3.23} 
\end{equation}
In particular,
\begin{align}
\begin{split} 
\|\cB_{\alpha,n,z}\|_{\cB(L^p(\bbR^n;d^nx))} \leq  
\big\|g_{\alpha,n}^{\vee}(z,\dott)\big\|_{L^1(\bbR^n;d^nx)},&     \lb{3.24} \\
\Re(\alpha) \in (0,\infty), \; z \in \bbC \backslash [0,\infty),&
\end{split}
\end{align}
and, if $\alpha \in (0,\infty)$, $z=-\lambda$, $\lambda \in (0,\infty)$, 
\begin{equation}
\big\|g_{\alpha,n}^{\vee}(-\lambda,\dott)\big\|_{L^1(\bbR^n;d^nx)}=(2\pi)^{-n/2}g_{\alpha,n}(-\lambda,0)=(2\pi)^{-n/2}(-\lambda)^{-\alpha}.   \lb{3.25}
\end{equation}
\end{theorem}
\begin{proof} By \eqref{2.30}, the Bessel potential operator $\cB_{\alpha,n,z}$  coincides with  the $- \alpha$-th power of the Helmholtz Laplacian $(- \Delta - z I_{L^2(\bbR^n;d^nx)})^{-\alpha}$ in $L^2(\bbR^n;d^nx)$,
\begin{equation}
\cB_{\alpha,n,z} = (H_0 - z I_{L^2(\bbR^n;d^nx)})^{- \alpha}, \quad \Re(\alpha) \in (0,\infty), \, z \in \bbC \backslash [0,\infty).
\end{equation}

Inequality \eqref{3.23} is Young's inequality \eqref{A.3} with $p=1$, $q=r=2$, $g = g_{\alpha,n}^{\vee}$, $h = f$. An application of Theorem \ref{tA.2} shows that, actually, equality holds in \eqref{3.23}, \eqref{3.24} for $\alpha \in (0,\infty)$, $z = - \lambda \in (-\infty,0)$, that is, 
\begin{align}
\begin{split} 
\|\cB_{\alpha,n,-\lambda}\|_{\cB(L^p(\bbR^n;d^nx))} = \big\|g_{\alpha,n}^{\vee}(-\lambda,\dott)\big\|_{L^1(\bbR^n;d^nx)},& \\
\alpha \in (0,\infty), \; \lambda \in (0,\infty).&
\end{split}
\end{align}
The fact \eqref{3.25} then follows from $g_{\alpha,n}^{\vee}(-\lambda,\dott)>0$ if $\alpha \in (0,\infty)$, $\lambda \in (0,\infty)$.	
\end{proof}

\section{Recovering the Riesz Composition Formula \\ from a Bessel Analog}  \lb{s4}

The fact that 
\begin{equation}
\begin{split}
(H_0 -z I_{L^2(\bbR^n;d^nx)})^{-\alpha} (H_0 -z I_{L^2(\bbR^n;d^nx)})^{-\beta} = (H_0 -z I_{L^2(\bbR^n;d^nx)})^{-\alpha - \beta},&     \lb{4.1} \\ 
\Re(\alpha), \Re(\beta), \in (0,\infty), \; z \in \bbC \backslash [0,\infty),&
\end{split}
\end{equation}
combined with \eqref{2.5} and \eqref{2.29} thus instantly yields the following result, a Bessel (resp., Helmholtz) analog of the Riesz composition formula (see \eqref{4.6}):

\begin{theorem} \lb{t3.3}
Suppose that $n\in \bbN$. Then for $\alpha, \beta \in \bbC$, $\Re(\alpha), \Re(\beta) \in (0,\infty)$, 
$z \in \bbC \backslash [0,\infty)$, 
\begin{align}
\begin{split} 
& \int_{\R^n} d^n x' \, \frac{K_{(n/2)-\alpha}\big((-z)^{1/2}|x-x'|\big) K_{(n/2)-\beta}\big((-z)^{1/2}|x'-y|\big)}{|x-x'|^{(n/2)-\alpha} |x'-y|^{(n/2)-\beta}}     \\ 
& \quad = \frac{\kappa_{\alpha,\beta,n}}{(-z)^{n/4}}
\frac{K_{(n/2)-\alpha-\beta}\big((-z)^{1/2}|x-y|\big)}{|x-y|^{(n/2)-\alpha-\beta}}, \quad x, y, \in \R^n, \;  x \neq y,    \lb{4.2}
\end{split} 
\end{align}
where
\begin{equation} \lb{4.3}
\kappa_{\alpha,\beta,n}  
=\frac{(2 \pi)^{n/2} 2^{-1} \Gamma(\alpha)\Gamma(\beta)}{\Gamma(\alpha+\beta)}.
\end{equation}
\end{theorem}

We call \eqref{4.2} a Bessel composition formula as fractional powers of the Laplacian, $-\Delta$, in connection with the Riesz potential and the Riesz composition formula \eqref{1.1}, are now replaced by those of $- \Delta -z$, $z \in \bbC \backslash [0,\infty)$ in connection with the Bessel potential. Formula \eqref{4.2} is mentioned, for instance, in \cite[p.~338]{Am22}, \cite[eq.~(2.5)]{AMS63}, \cite[eq.~(4, 7)]{AS61}, \cite[p.~135]{St70}.  The next result, the celebrated Riesz composition formula, is proved in detail by du Plessis  \cite[Theorem 3.1]{Du70}. We present an elementary and short argument based on the Bessel composition formula next. Indeed, taking $\lambda \downarrow 0$ in \eqref{4.2} and using the limiting form for $K_{\nu}$, 
\begin{equation} \lb{4.5}
K_{\nu}(\zeta) \underset{\zeta \rightarrow 0}{\sim}  
2^{-1} \Gamma(\nu) (\zeta/2)^{-\nu}, \quad \Re(\nu) > 0   
\end{equation}
(see, e.g., \cite[No.~9.6.9]{AS72}), one can now easily recover the Riesz composition formula for all relevant values of $\alpha$ and $\beta$ from the Bessel composition formula \eqref{4.2}.  In particular, one recovers the constant $k_{\alpha,\beta,n}$ in \eqref{4.7}, a task which, otherwise, requires some work (cf. \cite[Sect.~3.1, Lemmas~1--4]{Du70}), see however, Remark \ref{r4.4}. 

\begin{theorem}[The Riesz Composition Formula] \lb{t2.2}
Suppose that $n\in \bbN$. Then for $\alpha, \beta \in \bbC$, $\Re(\alpha), \Re(\beta), \Re(\alpha+\beta) \in (0, n/2)$, 
\begin{equation} \lb{4.6}
\int_{\R^n}  d^n x^{\p} \, |x-x^{\p}|^{2\alpha-n}|x^{\p}-y|^{2\beta-n} 
= k_{\alpha,\beta,n}|x-y|^{2\alpha+2\beta-n}, \quad x, y \in \bbR^n, \; x \neq y,
\end{equation}
where
\begin{equation} \lb{4.7}
k_{\alpha,\beta,n}=\pi^{n/2}\frac{\Gamma(\alpha)\Gamma(\beta)\Gamma((n/2)-\alpha-\beta)}
{\Gamma(\alpha+\beta)\Gamma((n/2)-\alpha)\Gamma((n/2)-\beta)} = \f{\gamma_{\alpha,n} \gamma_{\beta,n}}{\gamma_{\alpha + \beta,n}},
\end{equation}
and 
\begin{equation}
\gamma_{\alpha,n} = \pi^{n/2} 2^{2 \alpha} \Gamma(\alpha)/\Gamma((n/2)-\alpha).     \lb{4.8} 
\end{equation}
\end{theorem}
\begin{proof} 
Fixing $n\in \bbN\backslash\{1\}$, Theorem \ref{t3.3} implies 
\begin{align}
&\int_{\R^n} d^n x^{\p} \, \lambda^{(n-\alpha-\beta)/2} 
\frac{K_{(n/2)-\alpha}\big(\lambda^{1/2}|x-x^{\p}|\big) K_{(n/2)-\beta}\big(\lambda^{1/2}|x^{\p}-y|\big)}{|x-x^{\p}|^{(n/2)-\alpha} |x^{\p}-y|^{(n/2)-\beta}}    \no\\
&\quad=\kappa_{\alpha,\beta,n}\lambda^{(n-2\alpha-2\beta)/4}\frac{K_{(n/2)-\alpha-\beta}\big(\lambda^{1/2}|x-y|\big)}{|x-y|^{(n/2)-\alpha-\beta}},      \lb{4.9} \\
& \hspace*{2.45cm} \alpha, \beta, \lambda \in (0,\infty), \; x,y \in \R^n, \;   x \neq y.    \no
\end{align}
Restricting $\Re(\alpha), \Re(\beta), \Re(\alpha + \beta) \in (0,n/2)$, and using the estimate \eqref{2.31}, an application of Lebesgue's dominated convergence theorem together with the  
$\lambda\downarrow 0$ limiting behavior \eqref{4.5} in \eqref{4.9} then yields \eqref{4.6}--\eqref{4.8}. 
\end{proof}

\begin{remark}\lb{r3.5a}
One notes that the limit $z \to 0$, $z \in S_{\varepsilon}$, in \eqref{2.36} yields the integral, 
\begin{equation}
\int_0^{\infty} dt \, e^{-|x|^2 t/4} t^{[(n-2)/2 -\alpha]},
\end{equation}
and hence necessitates the condition $\alpha \in (0, n/2)$, and, analogously, necessitates 
$\beta \in (0, n/2)$ in Theorem \ref {t2.2}. Finiteness of the integral in \eqref{4.6} as $|x'| \to \infty$ then also requires 
$\alpha + \beta \in (0, n/2)$. \hfill $\diamond$
\end{remark}

\begin{remark}\lb{r4.4}
Following the kind suggestion by the anonymous referee, we now sketch an alternative and elementary proof of the Riesz composition formula: One starts with the identity,  
\begin{equation}
\Gamma((n/2) - \alpha) = \int_0^{\infty} ds \, s^{(n/2) - \alpha -1} e^{-s}, \quad 0 < \Re(\alpha) < n/2,
\end{equation}
makes the substitution $s = |x|^2/(4t)$, $x \in \bbR^n \backslash \{0\}$, $t \in (0,\infty)$, and obtains
\begin{equation}
\Gamma((n/2) - \alpha) |x|^{2 \alpha - n} = 4^{\alpha - (n/2)} \int_0^{\infty} dt \, t^{\alpha - (n/2) - 1} 
e^{- |x|^2/(4t)}, \quad 0 < \Re(\alpha) < n/2. 
\end{equation}
Introducing the Gauss--Weierstrass integral kernel
\begin{equation}
W(x,t) = (4\pi t)^{-n/2} e^{-|x|^2/(4t)}, \quad x \in \bbR^n, \; t \in (0,\infty),
\end{equation}
this results in 
\begin{align}
\begin{split}
\gamma_{\alpha,n}^{-1} |x|^{2\alpha -n} = \f{\Gamma((n/2)-\alpha)}{4^{\alpha} \pi^{n/2}  \Gamma(\alpha)} |x|^{2\alpha -n} = \f{1}{\Gamma(\alpha)} \int_0^{\infty} dt \, t^{\alpha-1} W(x,t),& 
\lb{4.13} \\ 
0 < \Re(\alpha) < n/2, \; x \in \bbR^n \backslash \{0\}.&
\end{split}
\end{align}
Identity \eqref{4.13} is also called {\it Bochner's subordination formula} for the Riesz potential kernel, see \cite[Theorem~2.34]{Sa02}. 

Next, introduce 
\begin{equation}
I= \gamma_{\alpha,n}^{-1} \gamma_{\beta,n}^{-1}\int_{\bbR^n} d^n y \, |y|^{2\alpha-n} |x-y|^{2\beta-n}      \quad \alpha, \beta \in (0,n/2). 
\end{equation}
By \eqref{4.13} and Fubini's theorem for nonnegative functions one obtains
\begin{align}
I &= \f{1}{\Gamma(\alpha)\Gamma(\beta)} \int_0^{\infty} dt \int_0^{\infty} ds \, t^{\alpha-1} t^{\beta-1} 
(W(\dott,t) * W(\dott,s))(x)    \no \\
&= \f{1}{\Gamma(\alpha)\Gamma(\beta)} \int_0^{\infty} dt \int_0^{\infty} ds \, t^{\alpha-1} t^{\beta-1} 
W(x,t+s),
\end{align}
employing $(W(\dott,t) * W(\dott,s))(x) = W(x,t+s)$. Substituting $t = u v$, $s = u(1-v)$, $u \in (0,\infty)$, $v \in (0,1)$, one arrives at
\begin{align}
I &=  \f{1}{\Gamma(\alpha)\Gamma(\beta)} \int_0^{\infty} u du \int_0^1 dv \, (uv)^{\alpha-1} [u(1-v)]^{\beta-1} W(x,u)     \no \\
&=  \f{1}{\Gamma(\alpha)\Gamma(\beta)} \int_0^{\infty} du \, u^{\alpha+\beta-1} W(x,u) \int_0^1 dv \, 
v^{\alpha-1} (1-v)^{\beta-1}    \no \\ 
&=  \f{1}{\Gamma(\alpha+\beta)} \int_0^{\infty} du \, u^{\alpha+\beta-1} W(x,u) \no \\
&= \gamma_{\alpha+\beta}^{-1} |x|^{2\alpha+2\beta-n}, \quad \alpha, \beta \in (0,n/2), \; \alpha + \beta \in (0,n/2),    \no 
\end{align}
employing the $\beta$ function integral and then Bochner's subordination formula \eqref{4.13} once more. Thus,  one recovers 
the Riesz composition formula. 

This computation extends to $\Re(\alpha), \Re(\beta) \in (0,n/2)$, $\Re(\alpha) + \Re(\beta) \in (0,n/2)$ using Fubini's theorem for signed functions. 

This argument is mentioned in Johnson \cite{Jo73} without proof. 
\hfill $\diamond$
\end{remark}

\begin{remark}\lb{r4.5}
The computation for the Riesz composition formula in Remark \ref{r4.4} also works for the Bessel composition formula (cf.\ \eqref{1.11} for the underlying operator version), in fact, this was derived earlier by Flett \cite{Fl71} before the analogous computation for the Riesz kernel. In this context one starts with the representation \eqref{2.37} 
\begin{align}
\begin{split} 
g_{\alpha,n}^{\vee}(z;x) = \f{1}{\Gamma(\alpha)} \int_0^{\infty} dt \, t^{\alpha_1} e^{zt} W(x,t),&    \\
\Re(\alpha), -\Re(z) \in (0,\infty), \; x \in \bbR^n \backslash \{0\},&  
\end{split} 
\end{align}
a Bochner subordination formula for the Bessel potential kernel. Then analogous computations as in Remark \ref{r4.4} yield 
\begin{align}
\begin{split}
\int_{\bbR^n} d^y \, g_{\alpha,n}^{\vee}(z;y) g_{\alpha,n}^{\vee}(z;x-y) = g_{\alpha,n}^{\vee}(z;x),&    \\
\Re(\alpha), -\Re(z) \in (0,\infty), \; x \in \bbR^n \backslash \{0\}.&     \lb{4.17}
\end{split}
\end{align}
But in this case simply focusing on the integral kernels of either side of the bounded operator identity \eqref{1.11} appear to be the quickest argument leading to \eqref{4.17}. 
\hfill $\diamond$
\end{remark}

\appendix
\section{Some Background on Convolution-type Operators} \lb{sA}

We start by denoting the unitary Fourier transform and its inverse in $L^2(\bbR^n;d^nx)$ by 
\begin{align}
(\cF f)(\xi) \equiv f^{\wedge}(\xi) = (2 \pi)^{-n/2} \LIM_{R\rightarrow \infty} \int_{|x| \leq R} d^n x \, 
f(x) e^{- i x \cdot \xi}, \quad f \in L^2(\bbR^n;d^nx),     \no \\
\big(\cF^{-1} g\big)(x) \equiv g^{\vee}(x) = (2 \pi)^{-n/2} \LIM_{R\rightarrow \infty} \int_{|\xi| \leq R} 
d^n \xi \, g(\xi) e^{i \xi \cdot x}, \quad g \in L^2(\bbR^n;d^n \xi),      \lb{A.1}
\end{align}
where $\LIM_{R\rightarrow \infty}$ denotes the limit in the norm of $L^2(\bbR^n;d^n \xi)$, respectively, in $L^2(\bbR^n;d^nx)$.

At this point we recall the following result from \cite[Theorem~IX.29]{RS75}:

\begin{theorem} \lb{tA.1} 
Assume $g \in L^{\infty}(\bbR^n;d^nx)$ and suppose that either \\[1mm] 
$(i)$ $g \in L^2(\bbR^n;d^nx)$, \\[1mm] 
or, \\[1mm] 
$(ii)$ $g^{\vee} \in L^1(\bbR^n;d^nx)$. \\[1mm] 
Then, in either case, one introduces 
\begin{equation}
(g(- i \nabla) f)(x) := (2 \pi)^{- n/2} \int_{\bbR^n} d^n y \, g^{\vee}(x-y) f(y),   
\quad f \in L^2(\bbR^n;d^nx),    \lb{A.2} 
\end{equation}
where the integral on the right-hand side of \eqref{A.2} converges for all $x \in \bbR^n$ in case $(i)$ 
and for $($Lebesgue\,$)$ a.e.~$x \in \bbR^n$ in case $(ii)$. 
\end{theorem}

In this context it is convenient to also recall Young's (convolution) inequality (see, e.g., \cite[p.~29]{RS75}),
\begin{align}
\begin{split} 
& \|g * h\|_{L^r(\bbR^n;d^nx)} \leq \|g\|_{L^p(\bbR^n;d^nx)} \|h\|_{L^q(\bbR^n;d^nx)},   \lb{A.3} \\
& \, g\in L^p(\bbR^n;d^nx), \; h \in L^q(\bbR^n;d^nx), \; 1 \leq p, q, r \leq \infty, \, p^{-1} + q^{-1} = 1 + r^{-1}.  
\end{split}
\end{align} 

Here $f\ast g$ denotes the convolution of the functions $f$ and $g$ defined by (see, e.g., \cite[Sect.~2.15]{LL01} and \cite[p.~6]{RS75})
\begin{equation}\lb{A.4}
(f\ast g)(x) = \int_{\bbR^n}d^ny\, f(x-y)g(y)\,\text{ for a.e.~$x\in \bbR^n$},
\end{equation}
provided that $f\in L^p(\bbR^n;d^nx)$ and $g\in L^q(\bbR^n;d^nx)$ for some $p,q\in [1,\infty)$ such that $p^{-1}+q^{-1}\geq 1$.

If $p\in [1,\infty)$ and $g\in L^1(\bbR^n;d^nx)$ are fixed, then the operator in $L^p(\bbR^n;d^nx)$ given by convolution with the function $g$,
\begin{equation}
A_gf = f\ast g,\quad f\in L^p(\bbR^n;d^nx),
\end{equation}
is a bounded operator by Young's inequality \eqref{A.3} (with $q=1$ and $r=p$) and one obtains the following estimate for the operator norm:
\begin{equation}\lb{A.6}
\|A_g\|_{\cB(L^p(\bbR^n;d^nx))} \leq \|g\|_{L^1(\bbR^n;d^nx)}.
\end{equation}
If $g$ is nonnegative a.e.~on $\bbR^n$, then equality actually holds in \eqref{A.6}. We did not find a statement of the following (likely, well-known) result in the literature, so we include its proof for completeness.
		
\begin{theorem}\lb{tA.2}
Let $n\in \bbN$, $p\in (1,\infty)\cup\{\infty\}$, and $g\in L^1(\bbR^n;d^nx)$ with $g\geq 0$ a.e.~on $\bbR^n$.  The operator $A_g:L^p(\bbR^n;d^nx)\to L^p(\bbR^n;d^nx)$ defined by
\begin{equation}
A_g f = f\ast g,\quad f\in L^p(\bbR^n;d^nx),
\end{equation}
is bounded and
\begin{equation}\lb{A.8}
\|A_g\|_{\cB(L^p(\bbR^n;d^nx))} = \|g\|_{L^1(\bbR^n;d^nx)}.
\end{equation}
In particular, $\|A_g\|_{\cB(L^p(\bbR^n;d^nx))}$ does not depend on $p\in (1,\infty)\cup\{\infty\}$.
\end{theorem}
\begin{proof}
If $p=\infty$, then Young's inequality \eqref{A.3} implies
\begin{equation}\lb{A.9}
\|A_gf\|_{L^{\infty}(\bbR^n;d^nx)} \leq \|g\|_{L^1(\bbR^n;d^nx)}\|f\|_{L^{\infty}(\bbR^n;d^nx)},\quad f\in L^{\infty}(\bbR^n;d^nx).
\end{equation}
Equality holds in \eqref{A.9} for $f=\chi_{\bbR^n}\in L^{\infty}(\bbR^n;d^nx)$ (here and throughout $\chi_{\Omega}$ denotes the characteristic function of a set $\Omega\subseteq \bbR^n$) as
\begin{align}
\|A_g\chi_{\bbR^n}\|_{L^{\infty}(\bbR^n;d^nx)} &= \underset{x\in \bbR^n}{\rm ess.sup} \, \bigg|\int_{\bbR^n}d^ny\, \chi_{\bbR^n}(x-y)g(y)\bigg|\\
&= \int_{\bbR^n}d^ny\, g(y) = \|g\|_{L^1(\bbR^n;d^nx)}.\no
\end{align}
Since $\|\chi_{\bbR^n}\|_{L^{\infty}(\bbR^n;d^nx)}=1$, \eqref{A.8} follows for $p=\infty$.

Next, let $p\in (1,\infty)$.  If $f\in L^p(\bbR^n;d^nx)$, then Young's inequality \eqref{A.3} implies $A_g f\in L^p(\bbR^n;d^nx)$ and
\begin{equation}
\|A_gf\|_{L^p(\bbR^n;d^nx)} = \|f \ast g\|_{L^p(\bbR^n;d^nx)} \leq \|g\|_{L^1(\bbR^n;d^nx)}\|f\|_{L^p(\bbR^n;d^nx)}.
\end{equation}
Therefore, $A_g$ is bounded with
\begin{equation}\lb{A.12}
\|A_g\|_{\cB(L^p(\bbR^n;d^nx))} \leq \|g\|_{L^1(\bbR^n;d^nx)}.
\end{equation}
If $\|g\|_{L^1(\bbR^n;d^nx)}=0$, then \eqref{A.8} is a trivial consequence of \eqref{A.12}.  To establish \eqref{A.8} in the nontrivial case $\|g\|_{L^1(\bbR^n;d^nx)}>0$, it suffices to find a sequence $\{f_j\}_{j=1}^{\infty}\subset L^p(\bbR^n;d^nx)$ such that $\|f_j\|_{L^p(\bbR^n;d^nx)}=1$ for all $j\in \bbN$ and
\begin{equation}\lb{A.13}
\lim_{j\to \infty}\|f_j\ast g\|_{L^p(\bbR^n;d^nx)} = \|g\|_{L^1(\bbR^n;d^nx)}.
\end{equation}
Consider the sequence $\{f_j\}_{j=1}^{\infty}\subset L^p(\bbR^n;d^nx)$ defined by
\begin{equation}
f_j = (2j)^{-n/p}\chi_{[-j,j]^n},\quad j\in \bbN.
\end{equation}
By inspection,  $\|f_j\|_{L^p(\bbR^n;d^nx)}=1$ for $j\in \bbN$.  Furthermore, the reflection symmetry of the hypercube $[-j,j]^n$, $j\in \bbN$, with respect to the origin $0\in \bbR^n$ in each of the $n$ coordinate directions implies:
\begin{align}
\|f_j\ast g\|_{L^p(\bbR^n;d^nx)}^p &= \int_{\bbR^n}d^nx\, \bigg|\int_{\bbR^n} d^ny\, f_j(x-y)g(y) \bigg|^p\no\\
&= \int_{\bbR^n}d^nx\, \bigg|\int_{\bbR^n}d^ny\, (2j)^{-n/p} \chi_{[-j,j]^n+x}(y)g(y) \bigg|^p\no\\
&= (2j)^{-n} \int_{\bbR^n}d^nx\, \bigg[\int_{[-j,j]^n+x}d^ny\, g(y)\bigg]^p,\quad j\in \bbN,\lb{A.15}
\end{align}
where in obvious notation
\begin{equation}
[-j,j]^n+x = \big\{y+x\,\big|\,y\in [-j,j]^n\big\},\quad x\in \bbR^n,\, j\in \bbN.
\end{equation}
Let $\varepsilon\in\big(0,\|g\|_{L^1(\bbR^n;d^nx)}\big)$.  The condition that $g$ is nonnegative a.e.~on $\bbR^n$ implies
\begin{equation}
\|g\|_{L^1(\bbR^n;d^nx)} = \int_{\bbR^n}d^ny\, g(y),
\end{equation}
so it is possible to choose $j(\varepsilon)\in \bbN$ such that
\begin{equation}
\int_{[-j(\varepsilon),j(\varepsilon)]^n}d^ny\, g(y) \geq \|g\|_{L^1(\bbR^n;d^nx)} - \frac{\varepsilon}{2}.
\end{equation}
For each $j>j(\varepsilon)$, one verifies that
\begin{equation}\lb{A.16}
[-j(\varepsilon),j(\varepsilon)]^n \subseteq [-j,j]^n+x\, \text{ whenever $x\in [-(j-j(\varepsilon)),j-j(\varepsilon)]^n$}.
\end{equation}
Combining \eqref{A.15} and \eqref{A.16}, one obtains:
\begin{align}
\|f_j\ast g\|_{L^p(\bbR^n;d^nx)}^p &\geq (2j)^{-n}\int_{[-(j-j(\varepsilon)),j-j(\varepsilon)]^n}d^nx\, \bigg[ \int_{[-j,j]^n+x}d^ny\, g(y)\bigg]^p\no\\
&\geq (2j)^{-n}\int_{[-(j-j(\varepsilon)),j-j(\varepsilon)]^n}d^nx\, \bigg[\int_{[-j(\varepsilon),j(\varepsilon)]^n}d^ny\, g(y)\bigg]^p\no\\
&\geq (2j)^{-n}\int_{[-(j-j(\varepsilon)),j-j(\varepsilon)]^n}d^nx\, \bigg[ \|g\|_{L^1(\bbR^n;d^nx)}-\frac{\varepsilon}{2}\bigg]^p\no\\
&= (2j)^{-n}\bigg[\|g\|_{L^1(\bbR^n;d^nx)}-\frac{\varepsilon}{2}\bigg]^p\int_{[-(j-j(\varepsilon)),j-j(\varepsilon)]^n}d^nx\no\\
&= \bigg(\frac{j-j(\varepsilon)}{j}\bigg)^n\bigg[\|g\|_{L^1(\bbR^n;d^nx)}-\frac{\varepsilon}{2}\bigg]^p,\quad j>j(\varepsilon),\,j\in \bbN.
\end{align}
As a result,
\begin{equation}
\|f_j\ast g\|_{L^p(\bbR^n;d^nx)} \geq \bigg(\frac{j-j(\varepsilon)}{j}\bigg)^{n/p}\bigg[\|g\|_{L^1(\bbR^n;d^nx)}-\frac{\varepsilon}{2}\bigg],\quad j>j(\varepsilon),\,j\in \bbN.
\end{equation}
The above analysis shows that for each fixed $\varepsilon\in \big(0,\|g\|_{L^1(\bbR^n;d^nx)}\big)$, there exists $j(\varepsilon)\in \bbN$ such that
\begin{align}
0&\geq \|f_j\ast g\|_{L^p(\bbR^n;d^nx)} - \|g\|_{L^1(\bbR^n;d^nx)}\no\\
&\geq \bigg[\bigg(\frac{j-j(\varepsilon)}{j}\bigg)^{n/p}-1\bigg]\|g\|_{L^1(\bbR^n;d^nx)} - \frac{\varepsilon}{2} \dott \underbrace{\bigg(\frac{j-j(\varepsilon)}{j}\bigg)^{n/p}}_{\leq 1}\no\\
&\geq \bigg[\bigg(\frac{j-j(\varepsilon)}{j}\bigg)^{n/p}-1\bigg]\|g\|_{L^1(\bbR^n;d^nx)} - \frac{\varepsilon}{2},\quad j>j(\varepsilon),\, j\in \bbN.\lb{A.19}
\end{align}
Since
\begin{equation}
\lim_{j\to \infty}\bigg[\bigg(\frac{j-j(\varepsilon)}{j}\bigg)^{n/p}-1\bigg]\|g\|_{L^1(\bbR^n;d^nx)} = 0,
\end{equation}
there exists $N(\varepsilon)\in \bbN$ with $N(\varepsilon)>j(\varepsilon)$ such that
\begin{equation}
\bigg[\bigg(\frac{j-j(\varepsilon)}{j}\bigg)^{n/p}-1\bigg]\|g\|_{L^1(\bbR^n;d^nx)} > -\frac{\varepsilon}{2},\quad j>N(\varepsilon),\, j\in \bbN.\lb{A.21}
\end{equation}
Thus, \eqref{A.19} and \eqref{A.21} yield the following result:  For each $\varepsilon\in \big(0,\|g\|_{L^1(\bbR^n;d^nx)}\big)$, there exists $N(\varepsilon)\in \bbN$ such that
\begin{equation}
-\varepsilon <   \|f_j\ast g\|_{L^p(\bbR^n;d^nx)} - \|g\|_{L^1(\bbR^n;d^nx)}  < \varepsilon,\quad j>N(\varepsilon),\, j\in \bbN.
\end{equation}
Hence, \eqref{A.13} holds.
\end{proof}

\begin{remark}
As a complement to Theorem \ref{tA.2}, we recall the following result for the product of a (pointwise) multiplication operator with a convolution operator in $L^2(\bbR^n;d^nx)$.  For $p\in (2,\infty)$, let $K_{f,g}$ denote the integral operator in $L^2(\bbR^n;d^nx)$ with integral kernel
\begin{equation}
k_{f,g}(x,y) = f(x)g(x-y)\,\text{ for a.e.~$(x,y)\in \bbR^n\times \bbR^n$}.
\end{equation}
An application of the Hausdorff--Young inequality and \cite[Theorem 4.1]{Si05} implies that $K_{f,g}\in\cB_p\big(L^2(\bbR^n;d^nx)\big)$ and
\begin{equation}
\|K_{f,g}\|_{\cB_p(L^2(\bbR^n;d^nx))} \leq \|f\|_{L^p(\bbR^n;d^nx)}\|g\|_{L^p(\bbR^n;d^nx)}.
\end{equation}
For further details, as well as extensions to more general integral operators, we refer to the remark given in \cite[p.~40]{Si05}.\hfill$\diamond$
\end{remark}

\medskip

\noindent 
{\bf Acknowledgments.} We are indebted to Andrei Martinez-Finkelshtein for very helpful hints to the literature. In addition, we are very grateful to the anonymous referee for the careful reading of our manuscript and for providing additional references and numerous suggestions that greatly improved the presentation throughout.  

 
\end{document}